\documentclass[12pt]{amsart}

\usepackage[english]{babel}
\usepackage[utf8]{inputenc}
\usepackage{amsmath}
\usepackage{lineno}
\usepackage{graphicx}
\usepackage[colorinlistoftodos]{todonotes}

\usepackage{tabu}
\usepackage{amsmath}
\usepackage{amssymb}
\usepackage{amsthm}
\usepackage{url}
\usepackage{graphicx}
\usepackage{caption}
\usepackage{subcaption}
\newtheorem{theorem}{Theorem}[section]

\newtheorem{corollary}[theorem]{Corollary}

\theoremstyle{remark}

\newtheorem{example}[theorem]{Example}

\usepackage{csquotes}
\begin{document}
\title{Random space and plane curves}

\author{Igor Rivin}

\date{\today}
\thanks{The author would like to thank the Institute for Advanced Study for its hospitality during the preparation of this paper. The results in this paper resolve a 30-odd year old question of Bill Thurston's - the author remains grateful to Thurston for this, and many other lessons. The author is grateful to Peter Sarnak, Curt McMullen, Robin Pemantle, and Stephen Wolfram for their comments.}
\address{St Andrews University School of Mathematics and Statistics and Temple University Mathematics Department.}
\email{igor.rivin@st-andrews.ac.uk}

\begin{abstract}
We study \emph{random knots}, which we define as a triple of random periodic functions (where a random function is a random trigonometric series, \[f(\theta) = \sum_{k=1}^\infty a_k \cos (k \theta) +b_k (\sin k \theta),\] with $a_k, b_k$ are independent gaussian random variables with mean $0$ and variance $\sigma(k)^2$ - our results will depend on the functional dependence of $\sigma$ on $k.$ In particular, we show that if $\sigma(k) = k^\alpha,$ with $\alpha < -3/2,$ then the probability of getting a knot type which admits a projection with $N$ crossings, decays at least as fast as $1/N.$ The constant $3/2$ is significant, because having $\alpha < -3/2$ is exactly the condition for $f(\theta)$ to be a $C^1$ function, so our class is precisely the class of random \emph{tame} knots.
\end{abstract}

\maketitle

\section{Introduction}

In this paper we study \emph{random curves in space and in the plane}. Our study was initially inspired by the subject of \emph{random knots.} The subject of random knots seems to be quite extensive, and many models of such knots have been suggested. The model in use in this paper, while not new (the author became aware of it in the early 1980s, when a version was suggested by Bill Thurston - indeed, in this paper we answer questions he posed back in the early 1980s), seems to be the least studied of all.

The model follows naturally from the following sequence of questions (and answers):
\begin{itemize}
\item{What is a knot?} A knot is a continuous map from $\mathbb{S}^1$ to $\mathbb{R}^3.$
\item{What is a continuous map from $\mathbb{S}^1$ to $\mathbb{R}^3?$} A map from $\mathbb{S}^1$ to $\mathbb{R}^3$ is a triple of continuous maps from $\mathbb{S}^1$ to $\mathbb{R}.$
\item{How do you represent a map from $\mathbb{S}^1$ to $\mathbb{R}?$} By a Fourier series: 
\[f(\theta)=\sum_{k=0}^\infty a_k \cos(k \theta) + b_k \sin(k \theta).\]
\item{What is a random such function?} The obvious way to produce a random periodic function is to let $a_k, b_k$ be random variables. The most obvious method: letting $a_k, b_k$ be independent identically distributed centered Gaussians does not work for our purposes (the functions thus obtained will be very wild). However, it is well-known that if $a_k, b_k$ decay at least as fast as $k^{3/2+\epsilon},$ then the resulting function will be of class $C^1.$\cite{kahane1993some}
\end{itemize}
Our main result is the following:
\begin{theorem}
\label{mainthm}
Suppose $a_k, b_k,c_k, d_k,\quad k=1, \dotsc, \infty$ are independent centered Gaussians with standard deviation of
$a_k, b_k$ equal to $k^{-3/2-\epsilon},$ for $\epsilon>0,$ and let \[x(\theta)=\sum_{k=1}^\infty a_k \cos(k \theta) + b_k \sin(k \theta),\] while \[y(\theta) = \sum_{k=1}^\infty c_k \cos(k \theta) + d_k \sin(k \theta).\] Then, the expected number of self-intersections of the plane curve $(x, y): \mathbb{S}^1\rightarrow \mathbb{R}^2$ is finite, and grows at most linearly in $1/\epsilon.$
\end{theorem}
Theorem \ref{mainthm} has a number of easy corollaries.
\begin{corollary}
\label{markov}
If $\gamma(\theta) = x(\theta), y(\theta),$ as in the statement of Theorem \ref{mainthm} is a random smooth plane curve, then the probability that the number of self-intersections of $\gamma$ exceeds $N$ decays at least linearly in $N.$
\end{corollary}
\begin{proof}
This is an immediate consequence of Theorem \ref{mainthm} combined with Markov's Inequality.
\end{proof}
\begin{corollary}
\label{decaycor}
Let $\gamma(\theta) = (x(\theta), y(\theta), z(\theta))$ be a random knot, where $x(\theta), y(\theta)$ are as in the statement of Theorem \ref{mainthm}, 
while 
\[z(\theta) = \sum_{k=1}^\infty e_k \cos(k \theta) + f_k \sin(k \theta),\] where $e_k, f_k$ are independent (also independent of all of the $a, b, c, d$ centered Gaussians with standard deviation of $e_k, f_k$ equal to
$k^{-3/2-\epsilon}$ (same as $a_k, b_k, c_k, d_k$) Then the probability that the crossing number of $\gamma$ is greater than $N$ decays at least linearly in $N.$
\end{corollary}
\begin{proof}
We take the first two coordinates of $\gamma$ and apply Corollary \ref{markov}.
\end{proof}
\subsection{Slowly decaying coefficients} A philosophically different model consists of taking \emph{random trigonometric polynomials,} as opposed to series (or, if the reader prefers, truncating the series after $N$ terms. The methods used to prove Theorem \ref{mainthm} immediately say that the expected number of crossings of a knot given by such a polynomial is \emph{quadratic} in $N,$ as long as the standard deviations of the coefficients decay slower than $k^{3/2}.$
\section{Proof of Theorem \ref{mainthm}}
Our first ingredient is \cite[Theorem 7.1]{edelman1995many}:
\begin{theorem}
\label{edelmank}
Let $v(t) = (f_0(t), \dotsc, f_n(t))^T$ be a differential function from $\mathbb{R}^m$ to $\mathbb{R}^{n+1},$ let $U$ be a measurable subset of $\mathbb{R}^m,$ and let $A$ be a random $m\times (n+1)$ matrix. Assume that the rows of $A$ are iid multivariate normal vectors, with mean zero and covariance matrix $C.$ The expected number of real roots of the system of equations $Av(t) = 0$ equals
\begin{equation}
\label{ekeq}
\pi^{-(m+1)/2}\Gamma\left(\frac{m+1}2\right) \int_U \left(\det\left[\frac{\partial^2}{\partial x_i \partial y_j}(\log v(x)^T C v(y))\left|_{x=y=t}\right.\right]\right)^{1/2}dt.
\end{equation}
\end{theorem}
Now, in Theorem \ref{mainthm} we are looking for the self-intersections of the curve $\gamma(\theta) = (x(\theta), y(\theta)).$ This means that we are looking for pairs $s, t$ such that $x(t) = x(s), y(t)=y(s),$ so we are looking for zeros of the vector function $(x(t)-x(s), y(t)-y(s).$ However, we want to eliminate the trivial zeros (where $t=s$), which can be achieved by looking for zeros of $(x(t)-x(s), y(t)-y(s))/g(t-s),$ where $g$ is a function vanishing to first order at zero, and nowhere else in $(0, 2\pi)$ (we will see that it is somewhat useful to have the flexibility of choosing $g.$

Now, we will apply Theorem \ref{edelmank} with \begin{gather*}f_{2k+1}(\vartheta, \varphi) = \frac{(\cos(k \vartheta)-\cos(k \varphi))}{g(\vartheta-\varphi)},\\ f_{2k}(\vartheta, \varphi) = \frac{(\sin(k \vartheta)-\sin(k \varphi))}{g(\vartheta -\varphi)},\end{gather*} while  \[A = \begin{pmatrix}a_1, b_1, \dotsc, a_k, b_k, \dotsc\\c_1, d_1, \dotsc, c_k, d_k, \dotsc\end{pmatrix}.\]

If $a_k, b_k, c_k, d_k$ are independent normal with standard deviation $\sigma_k = k^{-\alpha},$ the covariance matrix $C$ is the diagonal matrix, with diagonal entries $k^{-2\alpha}.$ Finally, simple trigonometry tells us that\begin{multline*}
v(x, y) C v(z, w) = \\\sum_{k=1}^\infty \frac{\cos(k(x-z)) + \cos(k(y-w)) - \cos(k(y-z)) - \cos(k(x-w))}{g(x-y)g(z-w)k^{2\alpha}}.
\end{multline*}
In the case where $\alpha$ is an integer, the above expression can be evaluated in closed form.
\begin{example}
Suppose $\alpha = 2.$ Then, it can be shown that:
\[
\sum_{k=1}^\infty \frac{\cos(k x)}{k^4} = -\frac{x^4}{48}+\frac{\pi  x^3}{12}-\frac{\pi ^2 x^2}{12}+\frac{\pi ^4}{90}.
\]
Using this, and $g(x) = x,$ one can write down the correlation in the case where $\alpha = 2$ as 
\begin{multline*}
v(x, y) C v(z, w) = \\
\frac{1}{24} (-2 w^2+3 w x+3 w y-2 w z-6 \pi  w-2 x^2-2 x y+3 x z+\\
6 \pi  x-2 y^2+3 y
   z+6 \pi  y-2 z^2-6 \pi  z-4 \pi ^2)
\end{multline*}
Using this, one can compute the expectation in closed form as:
\[
-4 \sqrt{3} \Gamma\left(3/2\right) \pi^{-3/2} \left(\text{csch}^{-1}(1)-1\right) \approx 0.130804
\]
\end{example}
In the general case, we set $g(x) = \sin(x/2).$ 
This gives us:
\begin{multline*}
\log(v(x, y) C v(z, w)) = \\\log \sum_{k=1}^\infty \frac{\cos(k(x-z)) + \cos(k(y-w)) - \cos(k(y-z)) - \cos(k(x-w))}{k^{2\alpha}} \\- \log(\sin((x-y)/2) \sin((z-w)/2)) = \\ f_1(x, y, z, w) + f_2(x, y, z, w),
\end{multline*}
where $f_1$ and $f_2$ denote the two log terms in the formula.
We have 
\[
\frac{\partial^2 f_2}{\partial x\partial z} = \frac{\partial^2 f_2}{\partial x\partial w} = \frac{\partial^2 f_2}{\partial y\partial z} = \frac{\partial^2 f_2}{\partial y\partial w} = 0,
\]
The other derivatives are a little more tedious to calculate:
\begin{align*}
\frac{\partial f_1}{\partial x} &= \frac{\sum_{k=1}^\infty \frac{-k\sin(k(x-z)) +  k\sin (k(x-w))}{k^{2\alpha}}}{\sum_{k=1}^\infty \frac{\cos(k(x-z)) + \cos(k(y-w)) - \cos(k(y-z)) - \cos(k(x-w))}{k^{2\alpha}}}\\
\frac{\partial f_1}{\partial z} &= \frac{\sum_{k=1}^\infty \frac{k\sin(k(x-z)) -  k\sin (k(y-z))}{k^{2\alpha}}}{\sum_{k=1}^\infty \frac{\cos(k(x-z)) + \cos(k(y-w)) - \cos(k(y-z)) - \cos(k(x-w))}{k^{2\alpha}}}.
\end{align*}
It is clear that the dominant terms in all the second partials (as $\alpha$ approaches $3/2$) are of the form
\[
\frac{\sum_{k=1}^\infty \frac{-k^2\cos(k(x-z))}{k^{2\alpha}}}{\sum_{k=1}^\infty \frac{\cos(k(x-z)) + \cos(k(y-w)) - \cos(k(y-z)) - \cos(k(x-w))}{k^{2\alpha}}},
\]
and these converge precisely when $\alpha > 3/2.$

 \section{The truth, and experiments} It is almost certain that Corollary \ref{decaycor} is nowhere near sharp, and the probability described therein decays \emph{exponentially} in $N.$ For space curves, we have conducted experiments with different decay rates, and estimated the knot types by computing the Alexander polynomial. When coefficients decay at the threshold decay rate of $k^{3/2},$ the results are that in $100$ experiments, we get $97$ unknots and $3$ trefoils. At the decay rate of $k^2,$ there are only unknots.
 At linear decay rate, around half the knots ($51$ out a hundred) are knots, $23$ are trefoils, $4$ are $5_2$ knots ("three twist knot"), $2$ $8_{20}$ knots and a smattering of others. For $\alpha=5/4,$ there are $89$ unknots, $8$ trefoils, and one each of figure $8,$ the $5_2$ knot, and one septafoil ($7_1$) knot.
 \subsection{Distribution of zeros of random Alexander polynomials}
 We generated 600 random knots of degree $100$ with centered normal coefficients decaying linearly, and we plotted their set of zeros (so, no multiplicity information is present). The highest degree of Alexander polynomial was equal to $12.$ The data is summarized in Figure \ref{alexfig}.
 \begin{figure}
\centering
\includegraphics[width=0.5\textwidth]{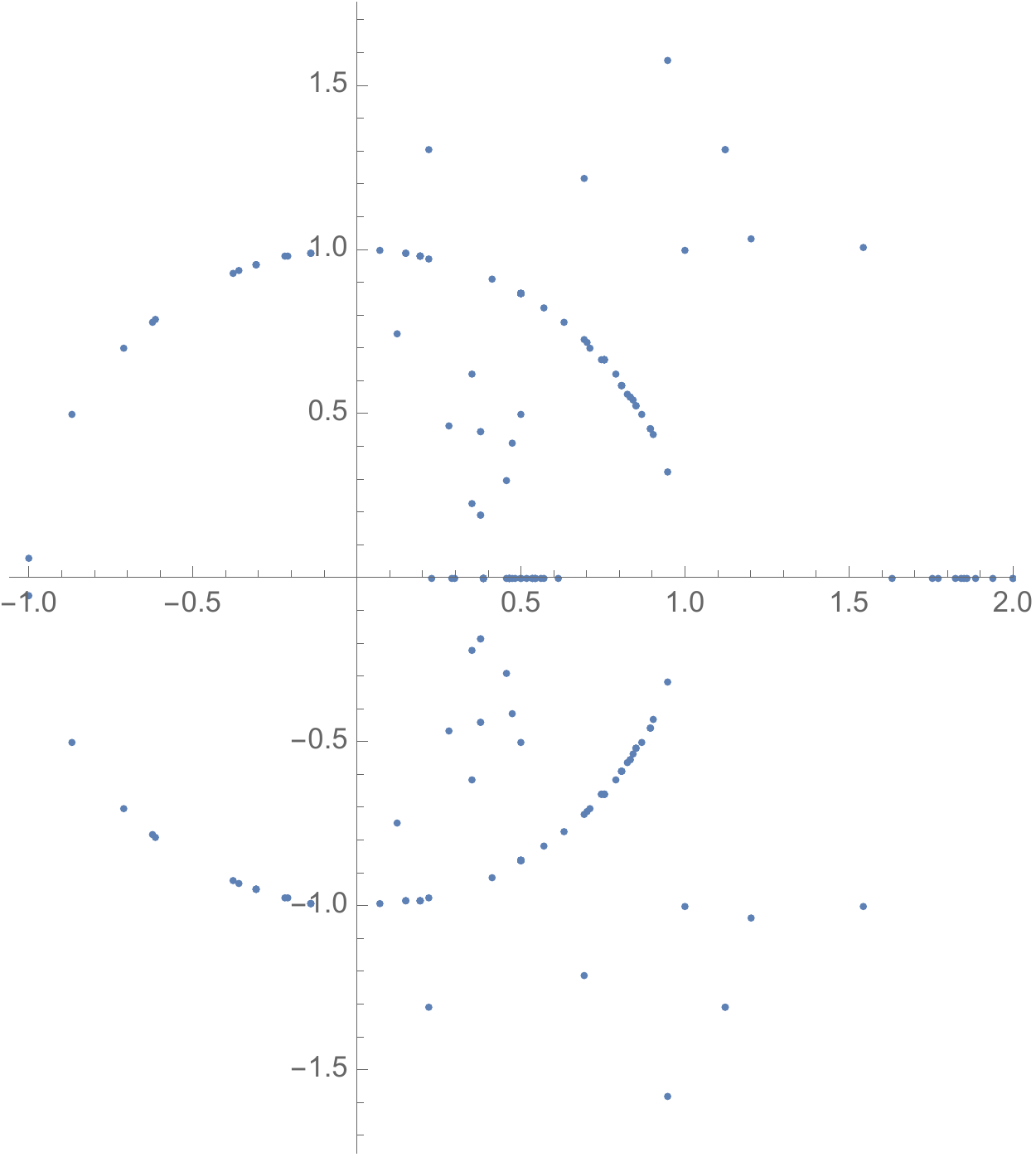}
\caption{\label{alexfig}Distribution of roots of Alexander polynomials of random knots.}
\end{figure}
You will notice that the distribution of the roots is quite asymmetric around the imaginary axis, quite unlike the distribution of zeros of random reciprocal polynomials (of degree 12, with coefficients uniform in $[-20, 20]$ - see Figure \ref{recipfig}. Note that \emph{all} the real roots are positive (in fact, the smallest one is bigger than about $0.22$).
\begin{figure}
\centering
\includegraphics[width=0.5\textwidth]{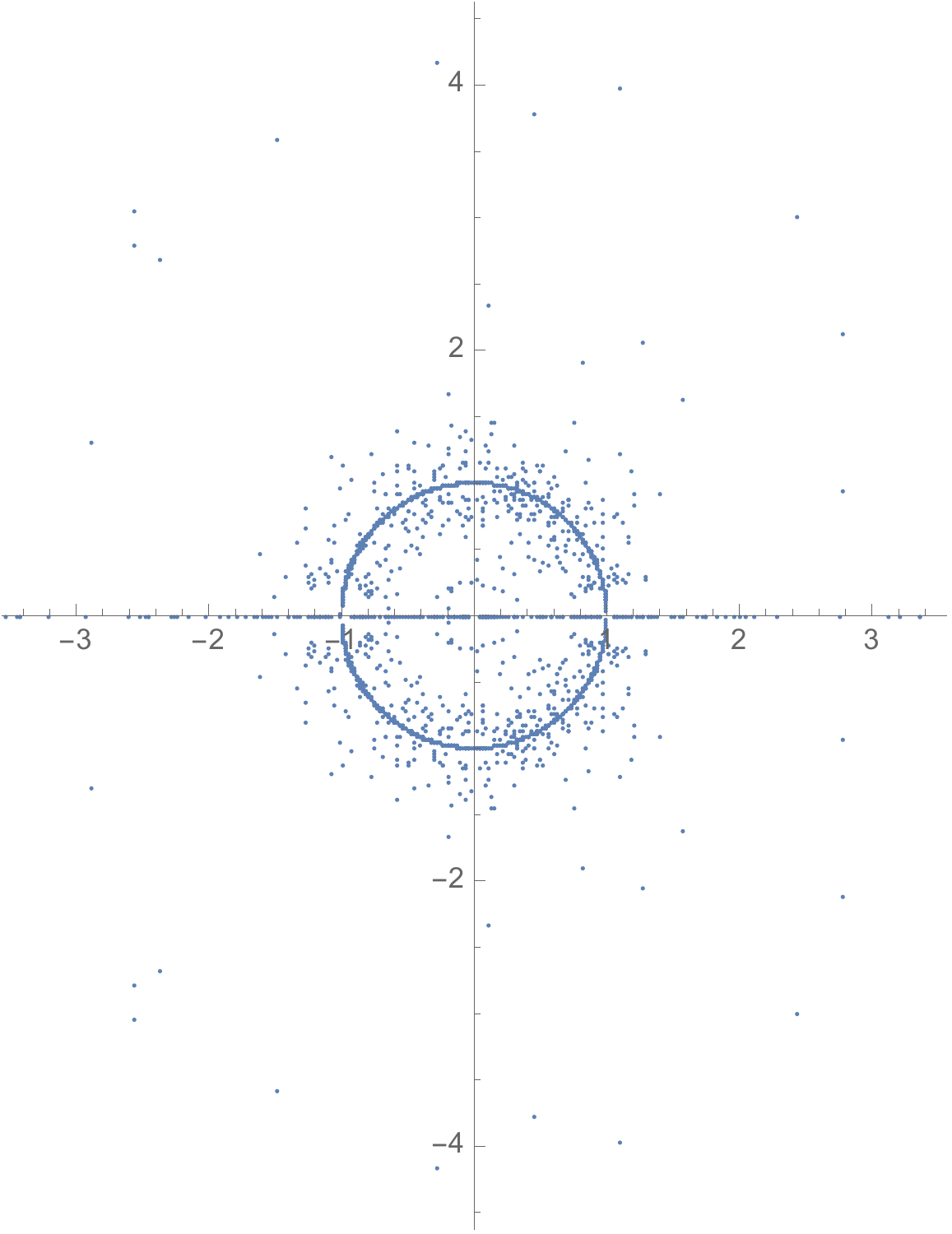}
\caption{\label{recipfig}Distribution of roots of random reciprocal polynomials.}
\end{figure}
Then, we generated $600$ roots of random polynomials of degree $30$ with iid centered normal coefficients. The root distribution now looks quite different (see Figure \ref{alexfig2}):
\begin{figure}
\centering
\includegraphics[width=0.5\textwidth]{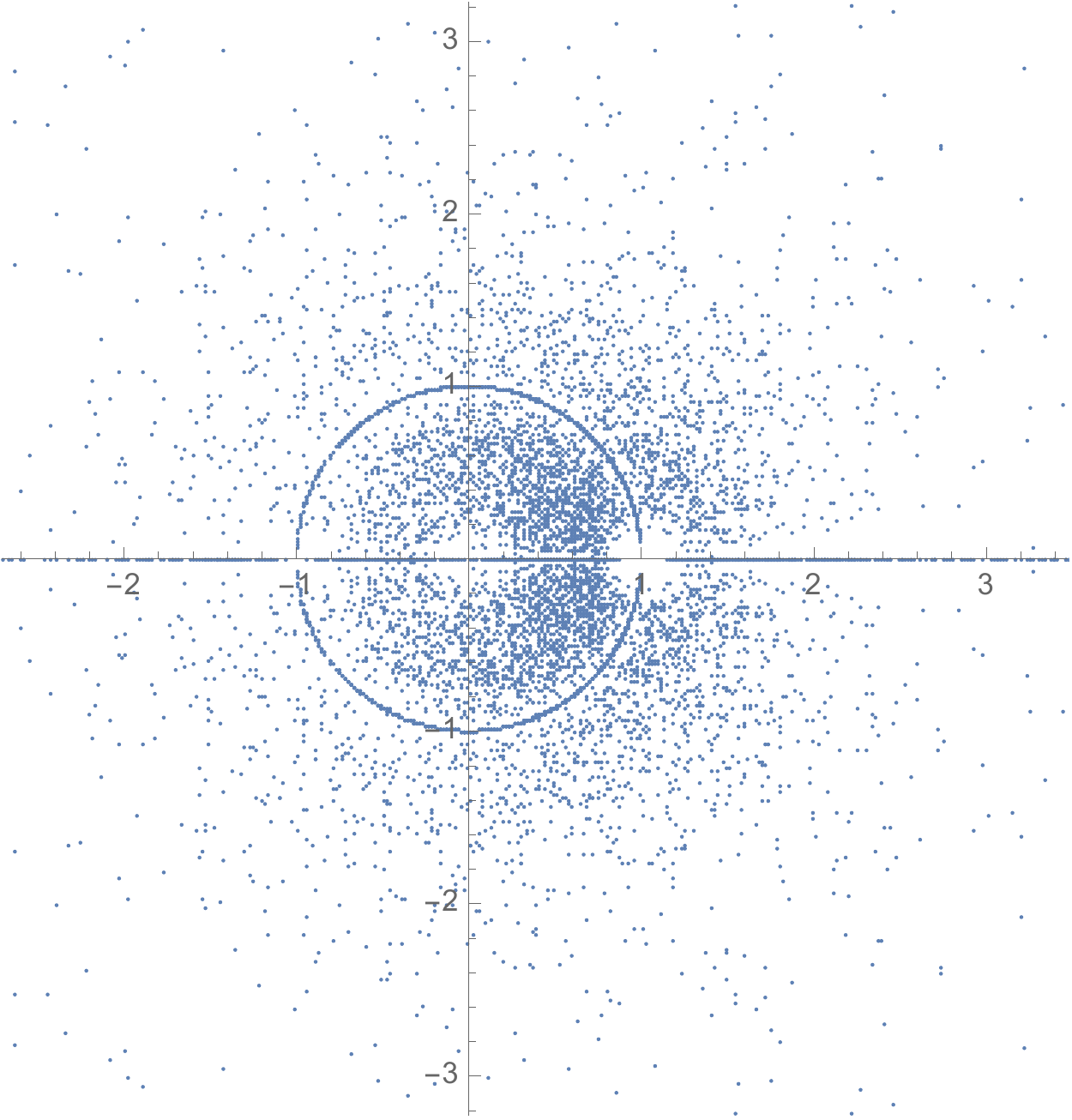}
\caption{\label{alexfig2}Distribution of roots of Alexander polynomials of random knots - no decay, degree $30.$}
\end{figure}
Now, for the strangest results of all, we look at Alexander polynomials of knots with non-decaying coefficients, and degree $60$ - see Figure \ref{alexfig3}, we used only $100$ knots here:
\begin{figure}
\centering
\includegraphics[width=0.5\textwidth]{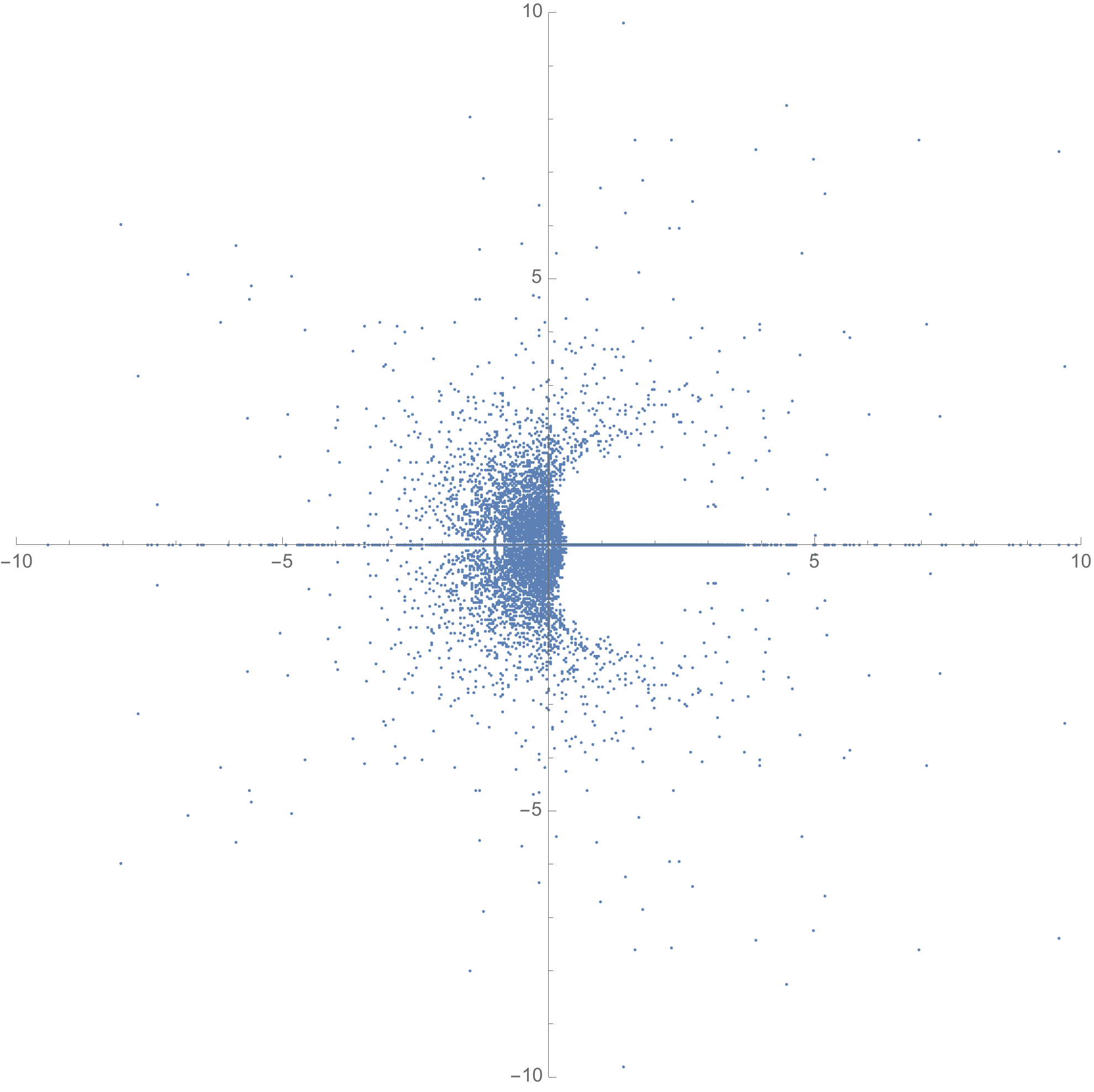}
\caption{\label{alexfig3}Distribution of roots of Alexander polynomials of random knots - no decay, degree $60.$}
\end{figure}
It looks like the hole around the point $1$ is widening and threatening to swallow the world. Interestingly, the vast majority ($99\%$) of Alexander polynomials in all cases have the sum of roots positive. On the other hand, while a random reciprocal polynomial has most of its roots on the unit circle, this is definitely \emph{not} true for Alexander polynomials in any of the regimes we tried.
\subsection{Coefficients of Alexander polynomials}
We see by looking at Figure \ref{alexfig3} (for example) that there are no zeros of Alexander polynomials of random knots around the point $z=1,$ from which, if the coefficients were positive, we would know (by the work of \cite{lebowitz2016central}) that the coefficients would be normally distributed.\footnote{The author would like to thank Robin Pemantle for bringing this work to his attention} Our Alexander polynomials are not blessed with positive coefficients, but below are the graphs (Figures \ref{alexroots60}, \ref{alexroots70}, \ref{alexroots80}, \ref{alexroots90}) of the logarithm of the absolute value of the coefficients as a function of the degree of the monomial. Each graph is for \emph{single random knot} - no averaging has been perfomed. It should also be noted that any reciprocal polynomial which evaluates to $1$ at $1$ is the Alexander polynomial of some knot (see \cite{kawauchi2012survey}), so there is definitely a concentration of measure phenomenon going on. The graphs appear to somewhat-smaller-than-semi ellipses. When I discussed this behaviour with Stephen Wolfram, he pointed out that very similar looking graphs appear in his book \emph{A New Kind of Science} \cite{wolfram2002new}, see Figure \ref{wolfig}
\begin{figure}
\centering
\includegraphics[width=0.5\textwidth]{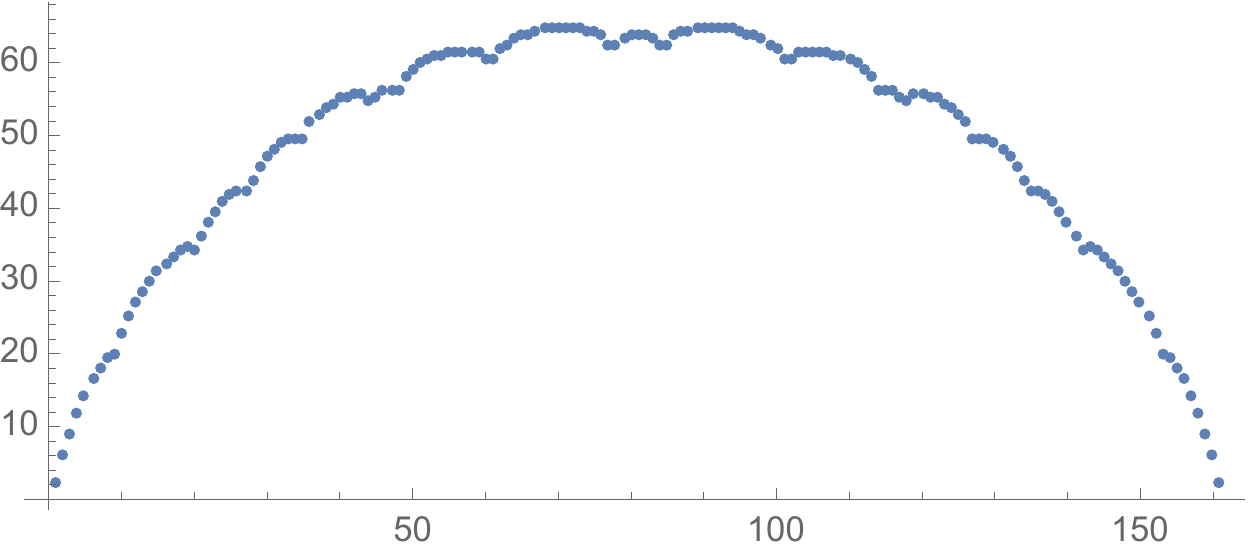}
\caption{\label{alexroots60}Distribution of logs of absolute values of the coefficients of a random Alexander polynomial of  degree $60,$ no decay}
\end{figure}
\begin{figure}
\centering
\includegraphics[width=0.5\textwidth]{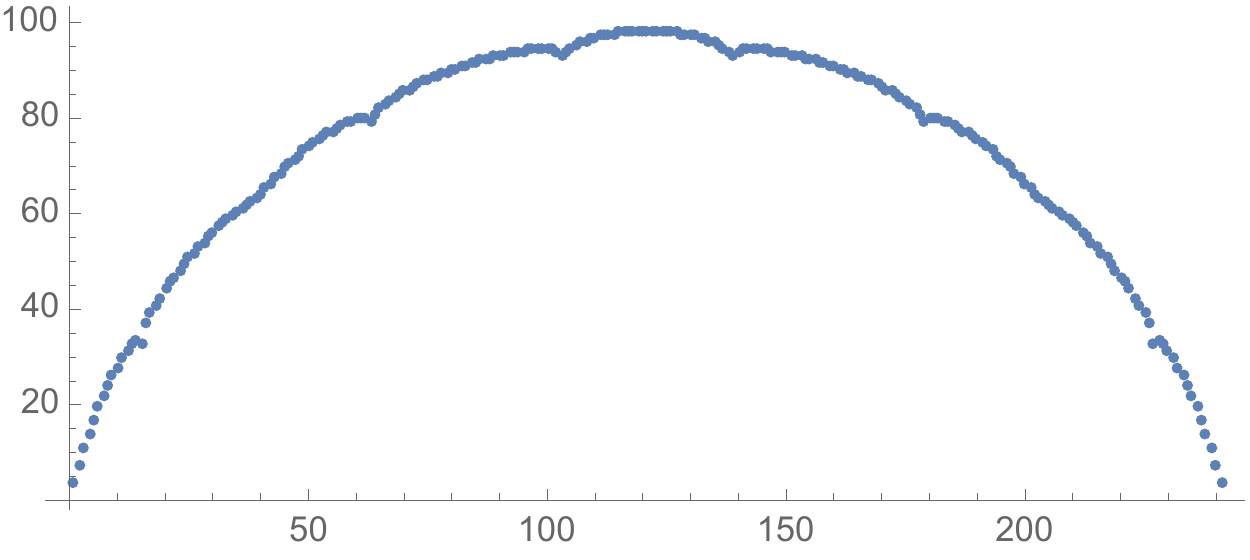}
\caption{\label{alexroots70}Distribution of logs of the absolute values of the coefficients of a random Alexander polynomial of  degree $70,$ no decay}
\end{figure}
\begin{figure}
\centering
\includegraphics[width=0.5\textwidth]{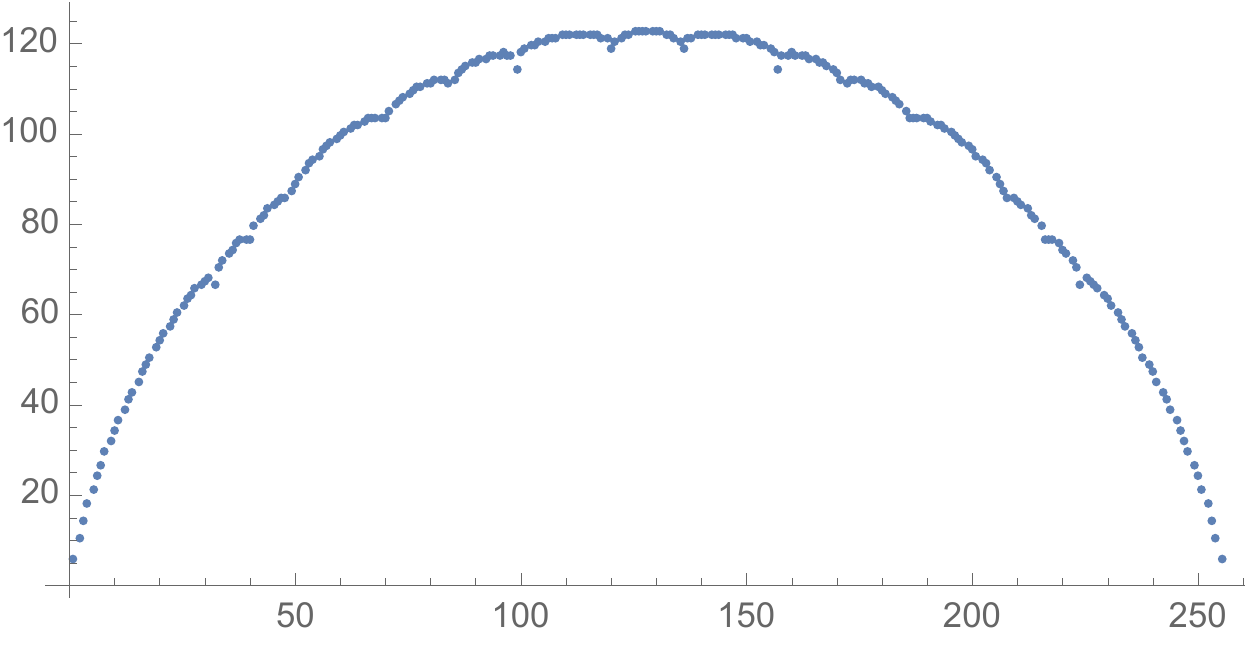}
\caption{\label{alexroots80}Distribution of logs of the absolute values of the coefficients of a random Alexander polynomial of  degree $80,$ no decay}
\end{figure}
\begin{figure}
\centering
\includegraphics[width=0.5\textwidth]{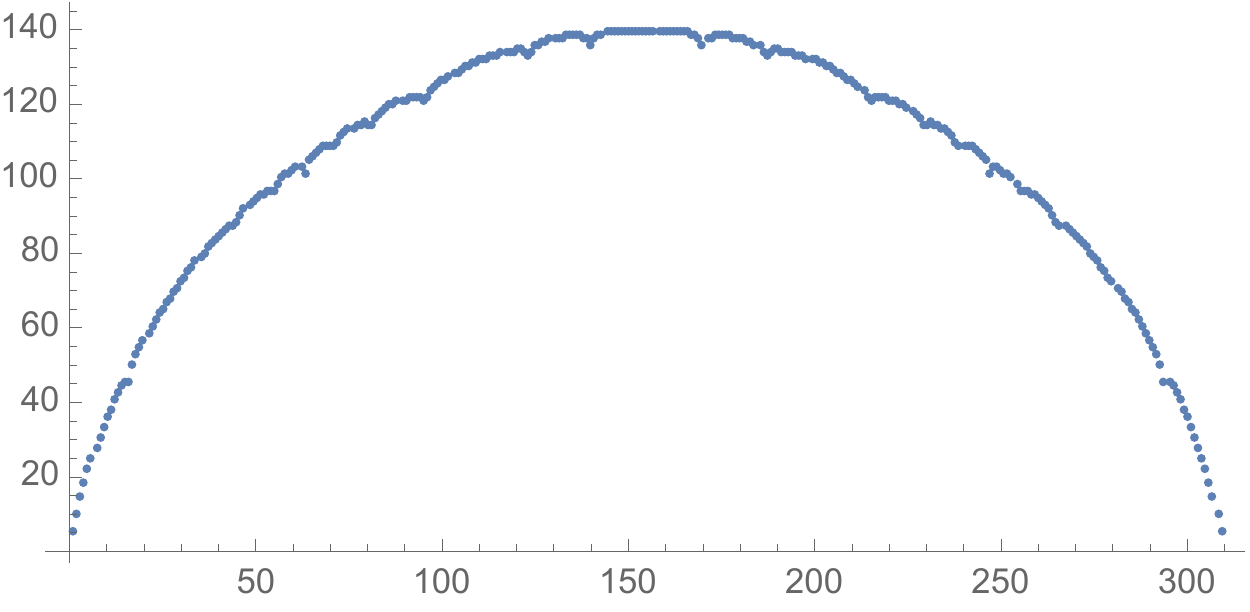}
\caption{\label{alexroots90}Distribution of logs of the absolute values of the coefficients of a random Alexander polynomial of  degree $90,$ no decay}
\end{figure}
\begin{figure}
\centering
\includegraphics[width=0.9\textwidth]{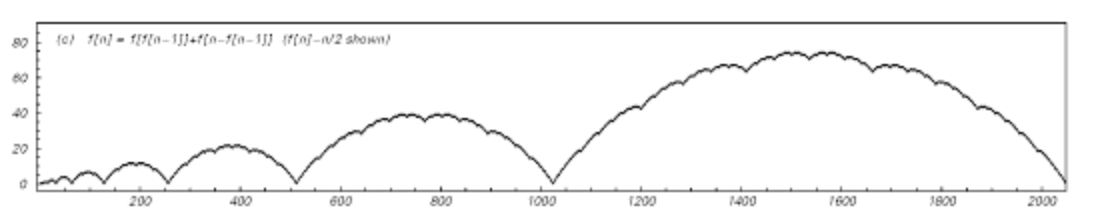}
\caption{\label{wolfig}A seemingly unrelated recursive sequence discovered by Stephen Wolfram}
\end{figure}
\section{Universality?}
The following model (call it $M_2$) is close to the one introduced in \cite{polyknots}: Pick $N$ points $p_1, \dotsc, p_N$ uniformly at random on the unit sphere $\mathbb{S}^2.$ Then, connect point $p_1$ to $p_2, $ point $p_2$ to $p_3,$ and finally $p_N$ to $p_1$ by straight line segments. The resulting closed curve will be almost surely non-self-intersecting, and so we can think of it as our random knot - this model, though not as natural as the one considered above has the advantages of being \emph{different} and also of being easy to model (the knot we have is already polygonal). A natural question is whether the distribution of the coefficients of Alexander polynomials of random Fourier knots bears any resemblance to the distribution we get from $M_2.$ Figures \ref{fig100},\ref{fig120},\ref{fig140},\ref{fig160},\ref{fig180} appear to answer that question unequivocally, and so the distribution of Alexander polynomials does not seem to strongly depend on the moment of random knots used.
\begin{figure}
\centering
\includegraphics[width=0.5\textwidth]{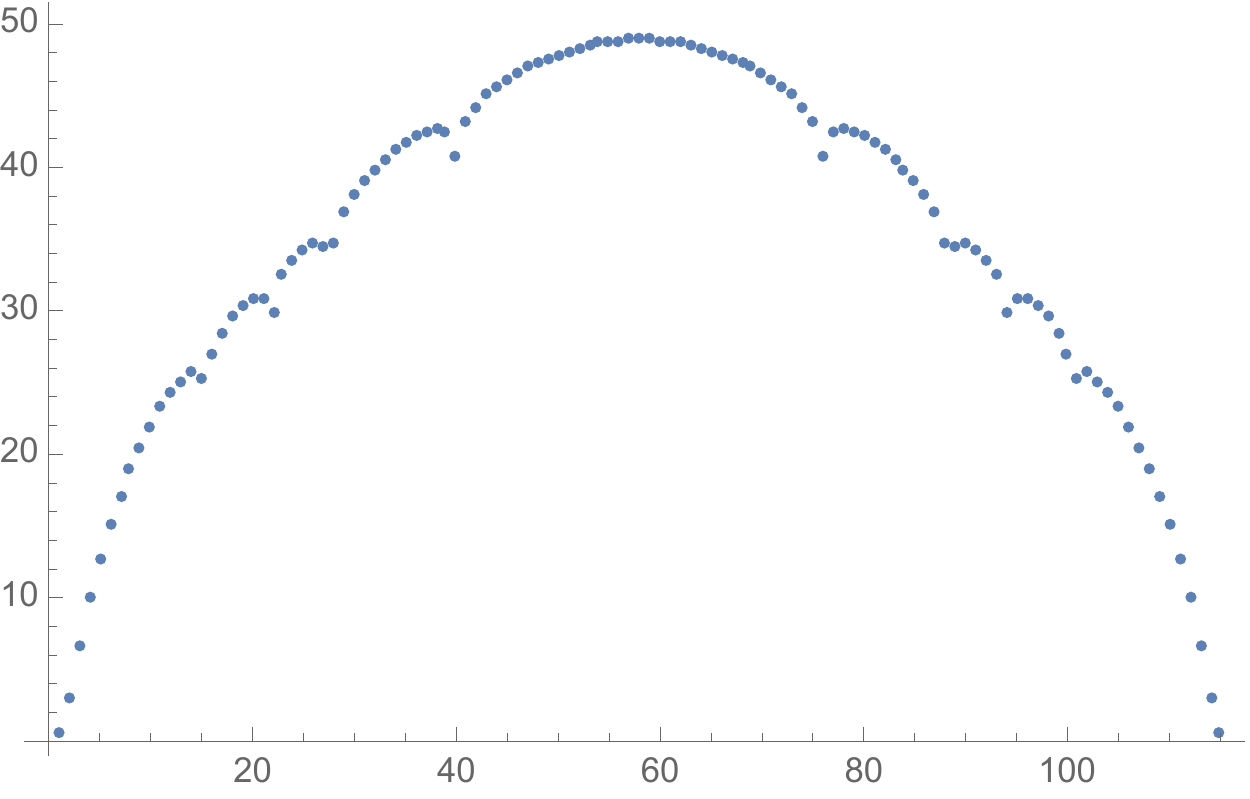}
\caption{\label{fig100}Distribution of logs of the absolute values of the coefficients of a random Alexander polynomial for $100$ random points on $\mathbb{S}^2$}
\end{figure}
\begin{figure}
\centering
\includegraphics[width=0.5\textwidth]{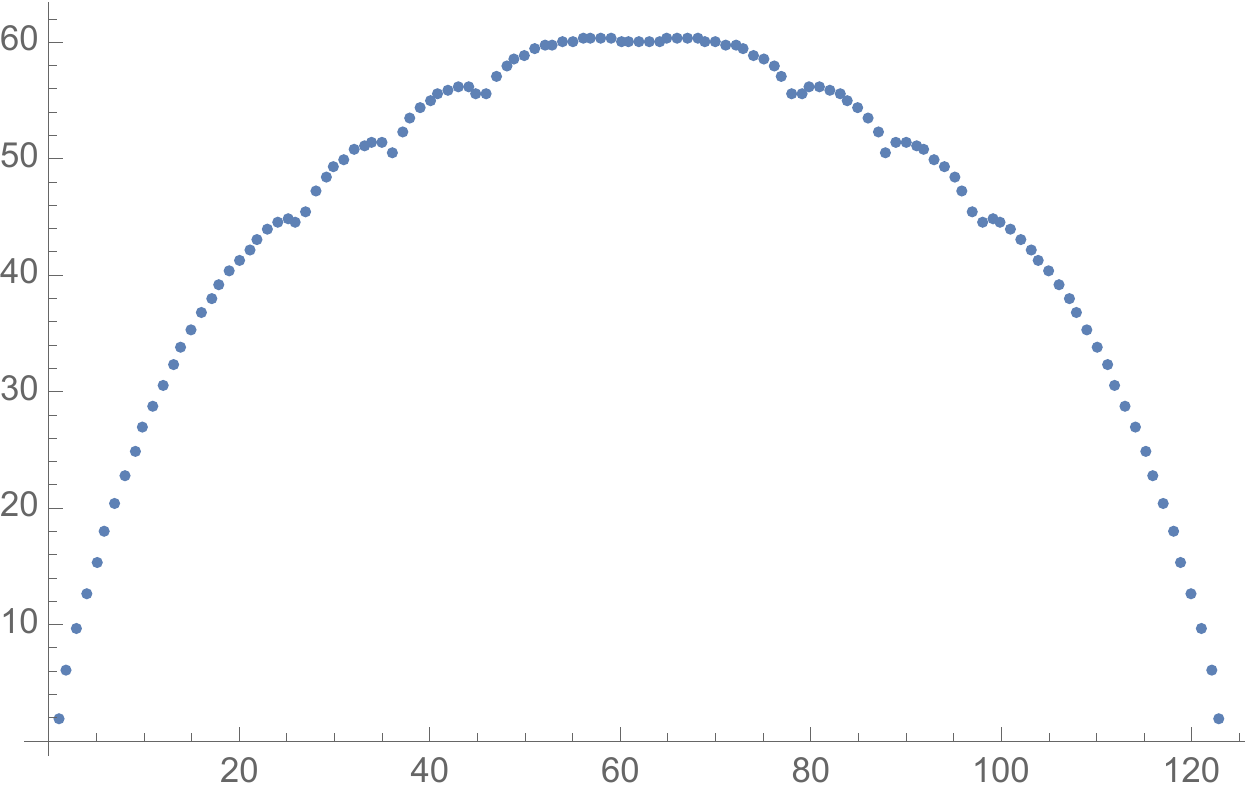}
\caption{\label{fig120}Distribution of logs of the absolute values of the coefficients of a random Alexander polynomial for $120$ random points on $\mathbb{S}^2$}
\end{figure}
\begin{figure}
\centering
\includegraphics[width=0.5\textwidth]{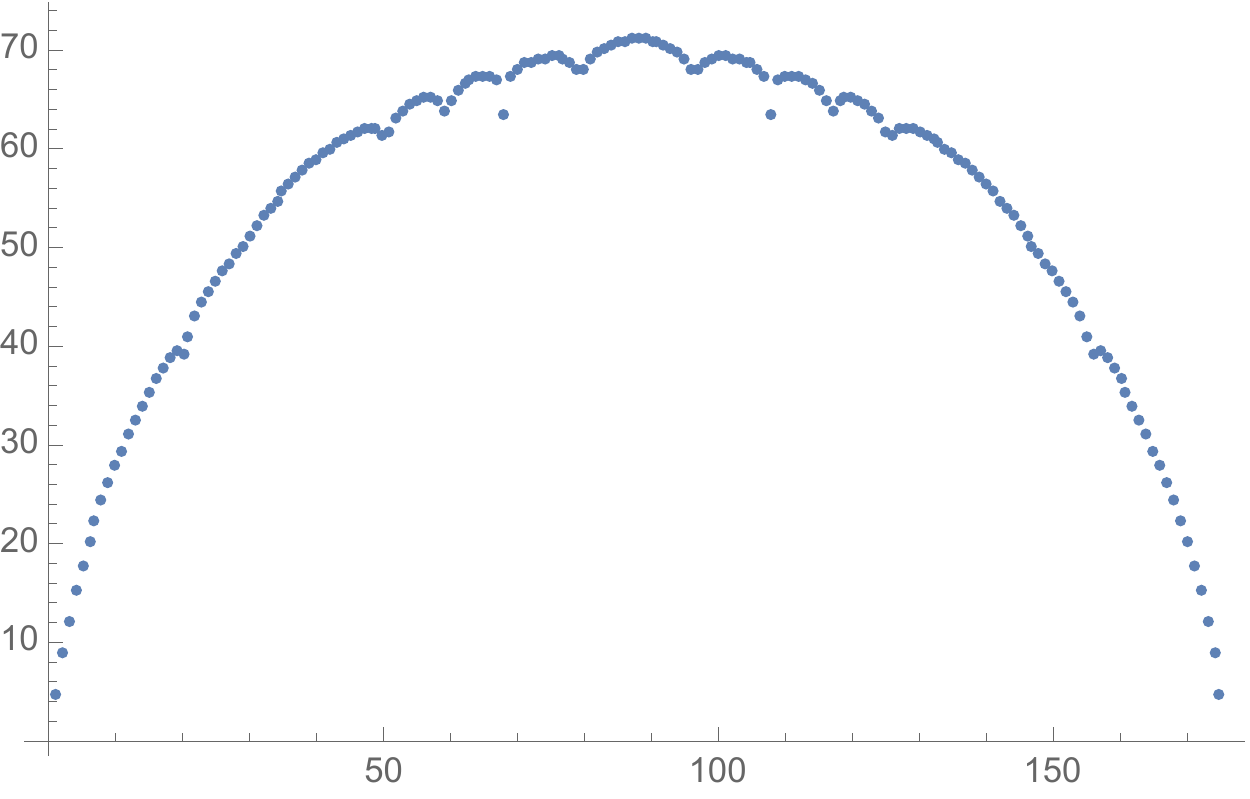}
\caption{\label{fig140}Distribution of logs of the absolute values of the coefficients of a random Alexander polynomial for $140$ random points on $\mathbb{S}^2$}
\end{figure}
\begin{figure}
\centering
\includegraphics[width=0.5\textwidth]{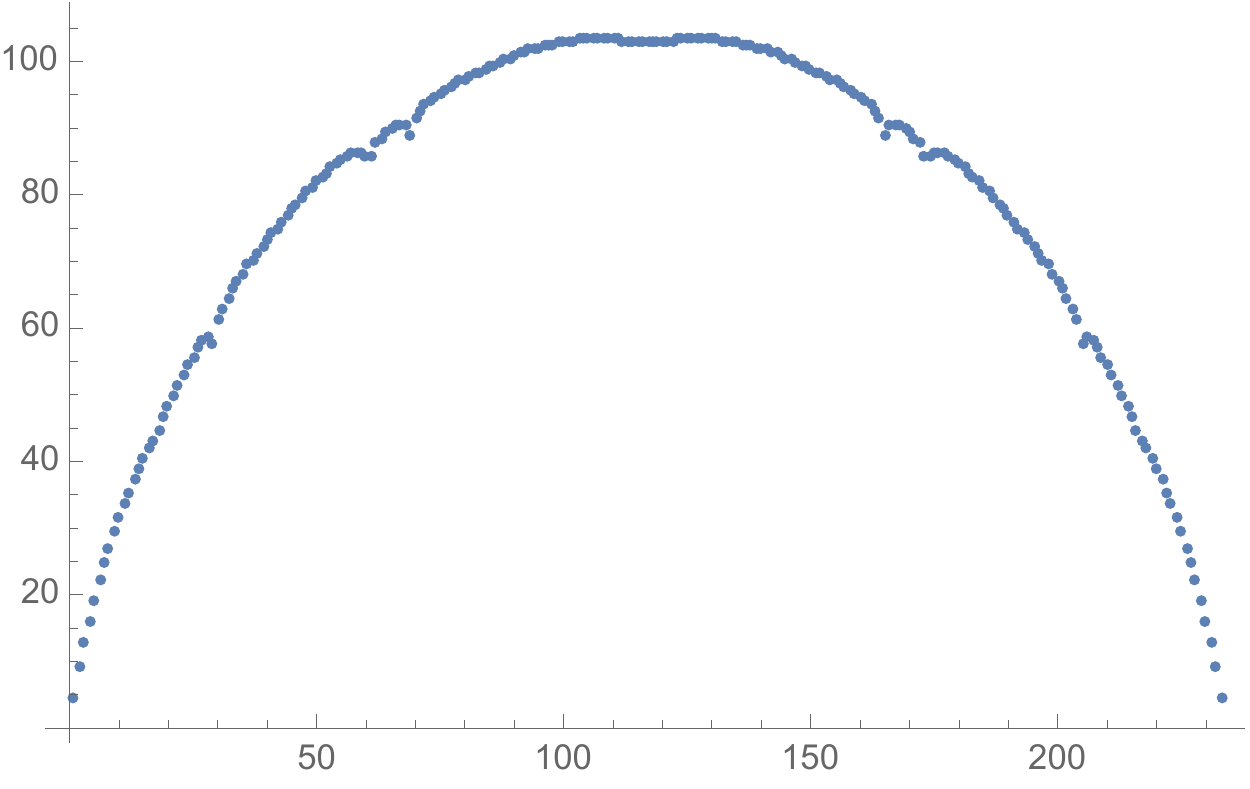}
\caption{\label{fig160}Distribution of logs of the absolute values of the coefficients of a random Alexander polynomial for $160$ random points on $\mathbb{S}^2$}
\end{figure}
\begin{figure}
\centering
\includegraphics[width=0.5\textwidth]{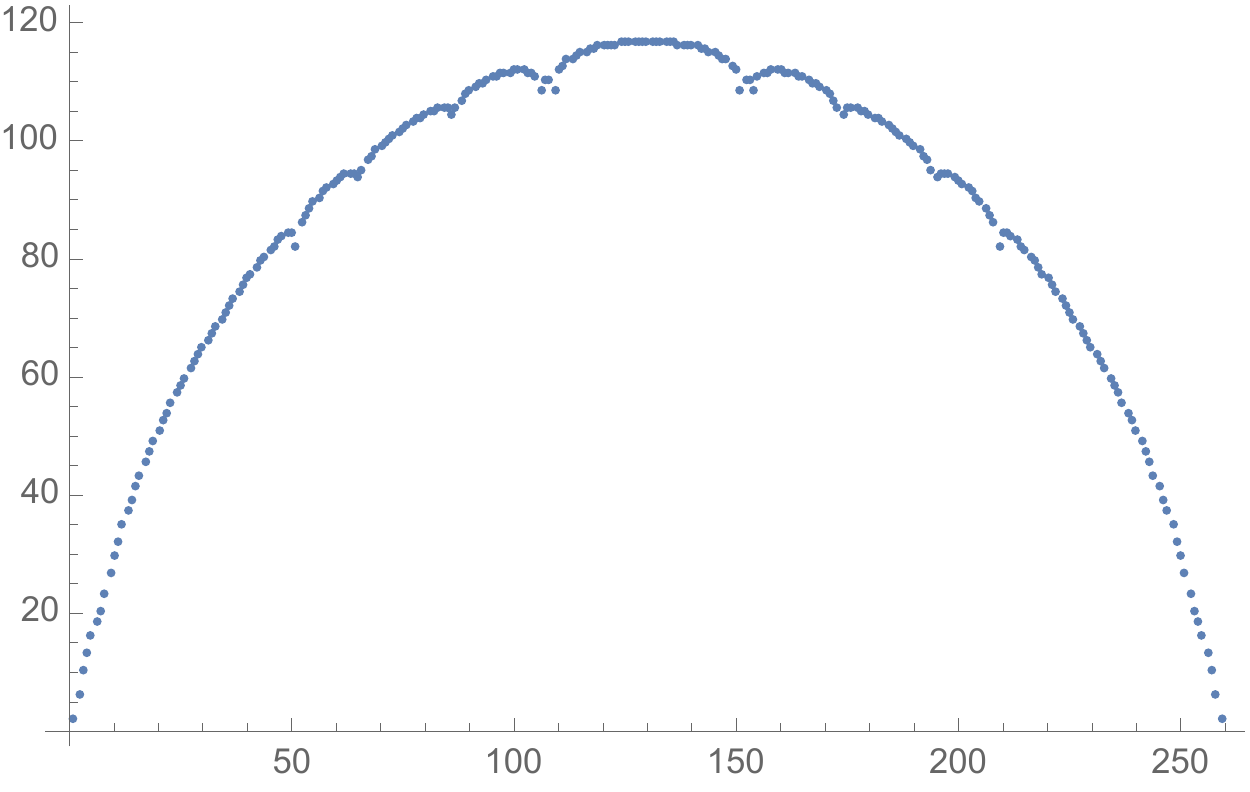}
\caption{\label{fig180}Distribution of logs of the absolute values of the coefficients of a random Alexander polynomial for $180$ random points on $\mathbb{S}^2$}
\end{figure}
\bibliographystyle{plain}
\bibliography{rivinbib}
 \end{document}